\documentclass{amsart}

\usepackage{
 amsmath, 
 amsxtra, 
 amsthm, 
 amssymb, 
 booktabs,
 comment,
 etex, 
 mathrsfs, 
 mathtools, 
 multirow,
 stmaryrd,
 tikz-cd, 
 bbm,
 xr}
\usepackage[all]{xy}
\usepackage{hyperref}

\newtheorem{theorem}{Theorem}[section]
\newtheorem{lemma}[theorem]{Lemma}

\newtheorem{corollary}[theorem]{Corollary}
\newtheorem{defn}[theorem]{Definition}

\newtheorem{lthm}{Theorem} % theorems with letters (for intro)

\theoremstyle{remark}
\newtheorem{remark}[theorem]{Remark}

\newcommand{\Gal}{\mathrm{Gal}}
\newcommand{\cG}{\mathcal{G}}
\newcommand{\SL}{\mathrm{SL}}
\newcommand{\ovp}{\overline{\varphi}}
\newcommand{\vp}{\varphi}
\newcommand{\GL}{\mathrm{GL}}
\newcommand{\Tr}{\mathrm{Tr}}
\newcommand{\Frob}{\mathrm{Frob}}
\newcommand{\cor}{\mathrm{cor}}
\newcommand{\ord}{\mathrm{ord}}
\newcommand{\ZZ}{\mathbb{Z}}
\newcommand{\CC}{\mathbb{C}}
\newcommand{\NN}{\mathbb{N}}
\newcommand{\QQ}{\mathbb{Q}}
\newcommand{\Fp}{\mathbb{F}_p}
\newcommand{\Zp}{\ZZ_p}

\usepackage[utf8]{inputenc}

\title[Ramanujan's tau and Mazur--Tate elements]{Congruences between Ramanujan's tau function and elliptic curves, and Mazur--Tate elements at additive primes}

\author[A. Doyon]{Anthony Doyon}
\address[Doyon]{D\'epartement de Math\'ematiques et de Statistique\\
Universit\'e Laval, Pavillion Alexandre-Vachon\\
1045 Avenue de la M\'edecine\\
Qu\'ebec, QC\\
Canada G1V 0A6}
\email{andoy32@ulaval.ca}

\author[A. Lei]{Antonio Lei}
\address[Lei]{D\'epartement de Math\'ematiques et de Statistique\\
Universit\'e Laval, Pavillion Alexandre-Vachon\\
1045 Avenue de la M\'edecine\\
Qu\'ebec, QC\\
Canada G1V 0A6}
\email{antonio.lei@mat.ulaval.ca}



\begin{document}
\begin{abstract}
 We show that if $E/\mathbb{Q}$ is an elliptic curve with a rational $p$-torsion for $p=2$ or $3$, then there is a congruence relation between Ramanujan's tau function and  $E$ modulo $p$. We make use of such congruences to compute the Iwasawa invariants of $2$-adic and $3$-adic Mazur--Tate elements attached to Ramanujan's tau function. We also investigate numerically the Iwasawa invariants of the Mazur--Tate elements attached to an elliptic curve with additive reduction at a fixed prime number.
\end{abstract}

\maketitle

\section{Introduction}

Let  $\Delta\in S_{12}(\SL_2(\ZZ))$ be the unique normalized cuspidal modular form of weight twelve and level one defined by Ramanujan's tau function $\tau:\NN\rightarrow \ZZ$, given explicitly by the $q$-expansion
\[
\Delta(z)=\sum_{n\ge1}\tau(n)q^n=q\prod_{n\ge1}\left(1-q^n\right)^{24},
\]
where $q=e^{2\pi i z}$ with $z$ being a variable in the upper half plane $\{z\in\mathbb{C}:\mathrm{Im}(z)>0\}$. This modular form encodes very rich arithmetic information and plays an important role in modern day Number Theory. Ramanujan's tau function satisfies interesting congruence relations, many of which can be explained by the theory of modular forms (see \cite{S-D1,S-D2,S-Dtau,serre-tau,serre-cong} for detailed discussions). One important feature in the theory of $p$-adic families of modular forms is congruence relations between Fourier coefficients of modular forms (see \cite{hida,serre-zeta,emerton}). It is therefore  natural  to study congruences between $\Delta$ and other modular forms. In \cite[\S6]{sujatha}, Sujatha discussed a congruence relation between $\Delta$ and the elliptic curve $X_0(11)$ modulo 11. We note in particular that $11\nmid \tau(11)= 	534612$, meaning that $11$ is an ordinary prime for $\Delta$. The congruence modulo $11$ above originates from the fact that both $\Delta$ and the weight-two modular form attached to $X_0(11)$ lie inside a Hida family. 

The starting point of the present article is to study congruence relations between $\Delta$ and other elliptic curves. Let $f_1$ and $f_2$ be two modular forms of weights $k_1$ and $k_2$ respectively. A necessary condition for the two modular forms to satisfy a congruence relation modulo a prime number $p$ is that $k_1\equiv k_2\mod p-1$. The fact that $\Delta$ is of weight $12$ means that if it satisfies a congruence relation with a weight two modular form, $p-1$ has to divide $10$. The only possible values $p$ can take are $2,3$ and $11$. We have $2\mid\tau(2)=-24$ and $3\mid\tau(3)=252$. In particular, $2$ and $3$ are both non-ordinary primes for $\Delta$. This means that any congruences between $\Delta$ and an elliptic curve modulo 2 or 3 cannot come from Hida Theory. Nonetheless, our first main result tells us that such congruences exist.

\begin{lthm}[Theorem~\ref{thm:cong}]\label{thm:A}
Let $p\in\{2,3\}$. Let $E/\QQ$ be an elliptic curve and denote its conductor by $N_E$. Suppose that $E$ admits a $p$-torsion point defined over $\QQ$. Then  $$a_\ell(E) \equiv \tau (\ell)\mod p$$  for all primes $\ell \nmid pN_E$.
\end{lthm}

Curiously, even though $\Delta$ is non-ordinary at $p\in\{2,3\}$, an elliptic curve $E$ admitting a rational $p$-torsion, as imposed by Theorem~\ref{thm:A} can be ordinary at $p$. Indeed, if $E$ has good reduction at $p$ and $p\big| |E(\mathbb{F}_p)|$, then  $a_p(E)\equiv 1\mod p$, meaning that $p$ is an anomalous ordinary prime for $E$.

For a fixed prime $p$ and a modular form $f$ with good reduction at $p$, we write $\theta_{n,f}$ for the Mazur--Tate element attached to $f$ over the sub-extension of the cyclotomic $\Zp$-extension of $\QQ$ of degree $p^n$. We shall write $\lambda(\theta_{n,f})$ for the $\lambda$-invariant of $\theta_{n,f}$  (see \S\ref{sec:MT} for a summary of the definitions of these objects). When $p$ is an ordinary prime for $f$, $\lambda(\theta_{n,f})$ is relatively well understood. We are interested in studying $\lambda(\theta_{n,\Delta})$ when $p$ is a non-ordinary prime for $\Delta$.  Using algorithms of Pollack, we have been able to calculate these $\lambda$-invariants explicitly for small $n$. Our numerical data suggest that they are given by $2^{n-2}-2$, $3^n-2$, $5^n-1$ and $7^n-1$ respectively; see Table~\ref{tab:table4} in \S\ref{sec:data} of the main body of the article.

 Pollack and Weston \cite{PW} have proved several formulae for $\lambda(\theta_{n,f})$ when $f$ is non-ordinary at $p$ with Serre weight 2 under various hypotheses on the residual representation, the weight, the prime number $p$ and the $p$-adic valuation of the $p$-th Fourier coefficient. One of the key ingredients in the work of Pollack--Weston is to compare $f$ to a weight 2 modular form $g$ via congruences modulo $p$. We may in fact describe the Iwasawa invariants of $\theta_{n,f}$ in terms of those of $\theta_{n,g}$, which  can be described explicitly. 

While the results of \cite{PW} do not apply to $\Delta$ at the {primes $p=2$ and $3$, the congruences modulo $p$ exhibited by Theorem~\ref{thm:A} suggest that some of the techniques in loc. cit. may allow us to study the Iwasawa invariants of  $\theta_{n,\Delta}$, shedding light on the regular patterns exhibited by the numerical data given in Table~\ref{tab:table4}. The following theorem where we compare the $\lambda$-invariants of $\theta_{n,\Delta}$ and $\theta_{n,E}:=\theta_{n,f_E}$, where $f_E$ is a weight-two modular form corresponding to an elliptic curve $E$ defined over $\QQ$ with conductor $32$ and $27$ (via the modularity theorem), is obtained along the lines of argument presented in \cite{PW}. We note that there exists a single isogeny class of such curves with four isomorphism classes. Three of the isomorphism classes of curves of conductor $27$ admit  a rational $3$-torsion and all four isogeny classes of curves of conductor $32$ admit a rational $2$-torsion, so Theorem~\ref{thm:A} applies for these curves. In fact, we obtain a full congruence between $E$ and $\Delta$ in the sense that 
\[
a_\ell(E)\equiv \tau(\ell)\mod p
\]
for all primes $\ell$ (including $\ell=p$).} This allows us to establish:

\begin{lthm}[Theorem~\ref{thm:Delta3}]\label{thm:B}
{Let $(p,N)=(2,32)$ or $(3,27)$, $n\ge1$ and $E$ an elliptic curve defined over $\QQ$ of conductor $N$.} If $\theta_{n,\Delta}\notin p\Zp[\cG_n]$, then
\[
\lambda(\theta_{n,\Delta})=\lambda(\theta_{n,E}).
\]
\end{lthm}

Note that the elliptic curves over $\QQ$ of conductor $32$ and $27$ have additive reduction at $2$ and $3$ respectively. It led us to study the following question.

\noindent\textbf{Question.} Is there a general formula for $\lambda(\theta_{n,E})$ if $E/\QQ$ is an elliptic curve with additive reduction at $p$?

We carry out numerical calculations of $\lambda(\theta_{n,E})$ at a prime $p\in\{2,3,5,7\}$, where $E$ has additive reduction, using Pollack's algorithm.  The codes we used are available on \url{https://github.com/anthonydoyon/Ramanujan-s-tau-and-MT-elts} and our results are presented in Tables~\ref{tab:table1}-\ref{tab:table3} in \S\ref{sec:data}. To our surprise, for a given $E$ and a given $p$, there seems to always exist a very regular formula for $\lambda(\theta_{n,E})$ in terms of $n$ when $n$ is sufficiently large. While we are not able to fully explain the origins of these formulae, we are able to explain why $\lambda(\theta_{n,E})$ is always at least $p^{n-1}$ (see Corollary~\ref{cor:lambda-add}). When $E$ has potentially good ordinary or potentially multiplicative reduction, Delbourgo \cite{del-compositio,del-JNT} has defined a $p$-adic $L$-function and  formulated an Iwasawa main conjecture for $E$ under the hypothesis that $E$ is the twist of a modular form with good reduction at $p$ by a Dirichlet character. We believe that our formulae on $\lambda(\theta_{n,E})$ should be related to Delbourgo's $p$-adic $L$-function. We intend to investigate this further in a future project.  For certain potentially supersingular elliptic curves, the formulae we find depends on the parity of $n$, which has a great resemblance of the formulae for elliptic curves that have good supersingular reduction at $p$. We plan to develop the supersingular analogue of Delbourgo's theory in order to better understand these formulae from a theoretical view-point.

\subsection*{Acknowledgement}
We thank Ashay Burungale, Henri Darmon, Daniel Delbourgo, Cédric Dion, Jeffrey Hatley, Chan-Ho Kim, Robert Pollack, Katharina M\"uller and Sujatha Ramdorai for interesting discussions during the preparation of this article. {We would also like to thank the anonymous referee for very helpful comments  on an earlier version of the article, which led to many improvements. In particular, we thank the referee for encouraging us to study the case $p=2$ in \S\ref{S:3MT}.} The authors' research is supported by the  NSERC Discovery Grants Program RGPIN-2020-04259 and RGPAS-2020-00096. Parts of this work were carried during a summer research project carried out by the first named author at Université Laval in 2020, which was supported by an NSERC Undergraduate Student Research Award.

\section{Congruences between Ramanujan's tau function  and elliptic curves with rational 2 or 3 torsions}
We study a congruence relation between $\Delta$ and elliptic curves defined over $\QQ$ which admit a rational $2$-torsion or a $3$-torsion. We begin with the following lemma on the values of $\tau$ modulo $2$ and modulo $3$.
\begin{lemma}\label{lem:tau23}
Let $p\in\{2,3\}$. For all primes $\ell\ne p$, we have
\[
  \tau(\ell) \equiv 1 +\ell\mod p
\]
\end{lemma}
\begin{proof}
Lehmer has proven that $\tau(\ell) \equiv 1 +\ell^{11} \mod 2^5$ if $\ell \neq 2$ and $\tau(\ell) \equiv \ell^{2} +\ell^{9} \mod 3^3$ if $\ell \neq 3$ (see for example \cite[\S2.1]{serre-tau}). But $\ell^{p-1}\equiv 1\mod p$ for all $\ell\ne p$ by Fermat's little theorem. Hence the result follows.
\end{proof}
We can now relate $\Delta$ to elliptic curves admitting a rational $2$-torsion or a rational $3$-torsion.

\begin{theorem}\label{thm:cong}
Let $p\in\{2,3\}$. Let $E/\QQ$ be an elliptic curve and denote its conductor by $N_E$. Suppose that $E$ admits a $p$-torsion point defined over $\QQ$. Then  $$a_\ell(E) \equiv \tau (\ell)\mod p$$  for all primes $\ell \nmid pN_E$.
\end{theorem}

\begin{proof}
Lemma~\ref{lem:tau23}  tells us that $\tau(\ell) \equiv 1 +\ell\mod p$ for all  $\ell \neq p$ . Since 
\[
a_\ell(E)=1+\ell-|E(\mathbb{F}_\ell)|,
\]
it suffices to show that $|E(\mathbb{F}_\ell)| \equiv 0\mod p$. Let $\alpha \in E(\QQ)[p]$ be a non-trivial $p$-torsion of $E$. Then,  $\langle\alpha\rangle$ is a subgroup of $E(\QQ)[p]$ of order $ p$. 

Consider the natural group homomorphism $\pi_\ell : E(\QQ) \to E(\mathbb{F}_\ell)$ given by reduction modulo $\ell$. By \cite[Proposition~VII.3.1]{Si}, $\pi_\ell$ induces an injective group homomorphism  $$E(\QQ)[p]\hookrightarrow E(\mathbb{F}_\ell).$$ In particular, $\pi_\ell(\langle\alpha\rangle)$ is a subgroup of order $p$ inside $E(\mathbb{F}_\ell)$. Thus, Lagrange's theorem tells us that $p\big | |E(\mathbb{F}_p)|$ as required.
\end{proof}

\begin{remark}
An alternative approach to prove Theorem~\ref{thm:cong} is to consider the Galois representation $\rho_{E,p}: G_\QQ\rightarrow \GL(E[p])=\GL_2(\ZZ/p\ZZ)$. Since $E$ admits a rational $p$-torsion, $\rho_{E,p}$ admits a one-dimensional trivial $\mathbb{F}_p$-linear sub-representation. Since the determinant of $\rho_{E,p}$ is given by  the mod $p$ cyclotomic character $\chi_p:G_\QQ\rightarrow (\ZZ/p\ZZ)^\times$, we have
\[
\rho_{E,p}\cong \begin{pmatrix}
1&*\\
0&\chi_p
\end{pmatrix}.
\]
Therefore, for all $\ell\nmid pN_E$, we have
\[
 a_\ell(E)\equiv \Tr(\rho_{E,p}(\Frob_\ell))=1+\chi_p(\Frob_\ell)=1+\ell \mod p,
\]
where $\Frob_\ell$ is the Frobenius at $\ell$.
\end{remark}

In the same vein as the results presented in \cite{PW}, Theorem~\ref{thm:cong} suggests that there might be a link between Iwasawa-theoretic objects of $\Delta$ and elliptic curves admitting a rational $2$-torsion or a $3$-torsion. In the next section, we shall review the objects we are interested in, namely Mazur--Tate elements attached to modular forms. We will then relate the 3-adic Mazur--Tate elements attached to $\Delta$ to certain elliptic curves admitting a rational $3$-torsion in \S\ref{S:3MT}.

\section{Review of   Mazur--Tate elements and Iwasawa invariants}\label{sec:MT}
\subsection{Definition and basic properties of Mazur--Tate elements}

In this section, we review the definition and some basic properties of the Mazur--Tate elements defined in \cite{MT}. We follow closely the exposition of Pollack--Weston in \cite[\S2.1 and \S2.2]{PW}. Throughout this section,  $p$ is a fixed prime number and $f \in S_{k}(\Gamma_0(N))$ is a fixed normalized cuspidal eigenform. For simplicity, we assume throughout that the Fourier coefficients of $f$ lie in $\ZZ$ (which is indeed the case when $f=\Delta$ or when $f$ corresponds to an elliptic curves defined over $\QQ$, which are the cases of interest in the present article). 

Suppose that $p$ is odd. Let $G_n =\Gal(\QQ (\mu_{p^n}) / \QQ)$. We identify an element $a \in (\ZZ / p^n \ZZ)^{\times}$ with the unique element $\sigma_a \in G_n$ satisfying $\sigma_a(\zeta) = \zeta^a$ for all $\zeta \in \mu_{p^n}$. Let $K_n$ be the unique sub-extension of $\QQ(\mu_{p^{n+1}})$ such that $[K_n:\QQ]=p^n$. We write $\cG_n=\Gal\left(K_n/\QQ\right)$, which we may identify with a quotient of $G_{n+1}$ and we are equipped with a natural projection map $\pi_n:G_{n+1}\twoheadrightarrow\cG_n$.

When $p=2$, we define $G_n$ to be $\Gal(\QQ(\mu_{p^{n+1}})/\QQ)$ and $\sigma_a\in G_n$ for $a\in(\ZZ/p^{n+2}\ZZ)^\times$  as before. Let $K_n$ be the fixed field of $\sigma_{-1}$ inside $\QQ(\mu_{p^{n+2}})$. Then, $K_n$ is a Galois extension of $\QQ$ of degree $p^n$ with $\Gal(K_n/\QQ)\cong \ZZ/p^n\ZZ$. We define $\cG_n$ and $\pi_n:G_{n+1}\twoheadrightarrow\cG_n$ as before.  

\begin{defn}\label{def:modsymb}
Let $\mathcal{R}$ be a commutative ring. We denote by $V_k (\mathcal{R})$ the space of homogenous polynomials of degree $k$ in two variables. Let $\Gamma \subset \SL_2(\ZZ)$ be a congruence subgroup. As in \cite[\S2.2]{PW}, we define a modular symbol $$\varphi_f\in H^1_c(\Gamma_0(N),V_{k-2}(\CC))\cong \mathrm{Hom}_{\Gamma_0(N)}\left(\mathrm{Div}^0(\mathbb{P}^1(\QQ),V_{k-2}(\CC))\right)$$ attached to $f$  satisfying $$\varphi_f(\{r\} - \{s\}) = 2 \pi i \int_s^r f(z)(zX+Y)^{k-2} dz$$
for $r,s\in\mathbb{P}^1(\QQ)$, where $\{r\}$ and $\{s\}$ are divisors associated to $r $ and $s$ respectively.
\end{defn}

\begin{defn}\label{def:MT}
Fix $n \in \NN$. When $p$ is odd, we define  $$\Theta_{n,f} = \sum_{a \in (\ZZ / p^{n+1} \ZZ)^\times}  \varphi_f (\{\infty\} - \{a/p^{n+1}\}) \big |_{(X,Y) = (0,1)} \cdot \sigma_a \in \CC[G_{n+1}]$$
and denote the image of $\Theta_{n,f}$ in $\CC[\cG_n]$ under the natural norm map induced by $\pi_n$ by $\tilde\Theta_{n,f}$. When $p=2$, $\Theta_{n,f}$ is defined similarly with $(\ZZ/p^{n+1}\ZZ)^\times$ and $a/p^{n+1}$ replaced by $(\ZZ/p^{n+2}\ZZ)^\times$ and $a/p^{n+2}$ respectively.

The $p$-adic \textbf{Mazur--Tate element} of level $n$ attached to $f$ is defined to be
\[
\theta_{n,f}=\frac{\tilde\Theta_{n,f}}{\Omega_f^+},
\]
where $\Omega_f^+$ is the cohomological period for $f$ given in \cite[Definition~2.1]{PW}.
\end{defn}

\begin{remark}
We are only looking at the $+1$-eigenspace of the involution induced by $\begin{pmatrix}-1&0\\0&1\end{pmatrix}$ on the space of modular {symbols} since this is where our numerical calculations will be carried out. This is why we only make use of the period $\Omega_f^+$ in the Definition~\ref{def:MT}.
Furthermore, as explained in \cite[Remark~2.2]{PW}, the choice of $\Omega_f^+$ ensures that $\theta_{n,f}\in\Zp[\cG_n]$.
\end{remark}

By the modularity theorem, if $E/\QQ$ is an elliptic curve of conductor $N_E$, then its $L$-function coincides with a unique normalized eigenform $f_E\in S_2(\Gamma_0(N_E))$. We let $\theta_{n,E}$ denote the Mazur--Tate element  $\theta_{n,f_E}$.

We now recall  the definitions of  Iwasawa $\mu$ and $\lambda$ invariants attached to $\theta_{n,f}$. For further discussion on this topic, we invite the reader to consult \cite[\S4]{pollack05} or \cite[\S3.1]{PW}. %Let $\mathcal{O}$ be a finite integrally closed extension of $\ZZ_p$.
%We consider the Iwasawa algebra $\Lambda = \underset{n}{\varprojlim}\ \Zp[G_n]$. Fix an isomorphism $\Lambda \cong \Zp \llbracket X \rrbracket$. Then every $L \in \Lambda$ admits an expansion $L = \sum_{i=0}^{\infty} a_i X^i$, with $a_i\in\Zp$.
Given an element $F\in \Zp[\cG_n]$, we choose a generator $\gamma_n$ of the Galois group $\cG_n$. We may write $F$ as a polynomial $\sum_{i=0}^{p^n-1}a_iX^i$, where $X=\gamma_n-1$.

\begin{defn}\label{def:mu-lambda}
For a non-zero element $F =\sum_{i=0}^{p^n-1} a_i X^i\in \Zp[\cG_n]$, we define the mu and lambda invariants of $F$ by
\begin{align*}
    \mu(F)& = \min\limits_{i} \ord_p (a_i),\\
    \lambda (F)& = \min \{i:\ord_p (a_i) = \mu (L) \},
\end{align*}
where $\ord_p$ denotes the $p$-adic valuation on $\ZZ$. When $F=0$, we set $$\mu(F)=\lambda(F)=\infty.$$
\end{defn}
\begin{remark}
 The definitions above are independent of the choice of the generator $\gamma_n$. 
\end{remark}
%We may choose $\gamma_n$ so that the image of $\gamma_{n+1}$ is $\gamma_n$ under the natural projection map $\cG_{n+1}$. 
We explain the strategy we use to compute these Iwasawa invariants for $\theta_{n,f}$ in the case $p$ is odd. When $p=2$, the strategy is very similar.

Explicitly, given 
\[
\Theta_{n,f}=\sum_{a \in (\ZZ / p^{n+1} \ZZ)^\times} C_a \cdot \sigma_a\in\CC[G_{n+1}],
\]
we obtain its projection $\tilde\Theta_{n,f}\in\CC[\cG_n]$ via
\[
\tilde\Theta_{n,f}=\sum_{a \in (\ZZ / p^{n+1} \ZZ)^\times} \frac{C_a}{\omega(a)} \cdot \pi_n(\sigma_a),
\]
where $\omega:G_n\rightarrow \Zp^\times$ is the Teichmüller character. Once we fix a generator $\gamma_n$ of $G_n$, we may write $\pi_n(\sigma_n)=(1+X)^{a'}$ for some integer $a'\in\{0,1,\ldots,p^n-1\}$. This gives
\[
\theta_{n,f}=\frac{1}{\Omega_f^+}\cdot\sum_{a \in (\ZZ / p^{n+1} \ZZ)^\times} \frac{C_a}{\omega(a)} \cdot (1+X)^{a'},
\]
which is the formula we use in our numerical calculations below.

When $\theta_{n,E}$ is non-zero, we may multiply $\theta_{n,E}$ by an appropriate constant so that its coefficients as a polynomial in $X$ are not divisible by $p$ simultaneously. If we write $\tilde{\theta}_{n,E}$ for this scaled polynomial, we may calculate $\lambda(\theta_{n,f})$ by finding the degree of the polynomial $\Fp[X]$ obtained from $\tilde{\theta}_{n,E}$ modulo $p$.

For the elliptic curves we consider in our calculations,  we may  work with a particular choice of $E'$ in the isogeny class containing $E$ where $\tilde{\theta}_{n,E'}=\theta_{n,E'}$. Indeed, $\lambda$-invariants are constant in an isogeny class and it is  conjectured that there always exists an $E'$ in any given isogeny class satisfying $\mu(\theta_{n,E'})=0$ (see \cite[Conjecture~1.11]{greenberg} in the good ordinary case and \cite[Conjecture~7.1]{PR} in the good supersingular case).

We are interested in $p$-adic Mazur-Tate elements mostly because they are closely related to the L-function. More precisely, $p$-adic Mazur-Tate elements satisfy the following interpolation property as pointed out in \cite[\S2]{PW}.

\begin{theorem}\label{thm:interpol}
Let $\chi$ be a  Dirichlet character factoring through $\cG_n$, but not $\cG_{n-1}$, where $n\ge1$ is an integer. Let $\theta_{n,f}$ be the $p$-adic Mazur-Tate element as defined in Definition~\ref{def:MT}. Then, $$\chi(\theta_{n,f}) = \tau(\chi) \frac{L(f, \chi^{-1}, 1)}{\Omega^{+}_f},$$ where $\tau(\chi)$ is the Gauss sum of $\chi$.
\end{theorem}

\begin{proof}
See \cite[\S2]{PW} and \cite[\S8]{MT}. Note that $\chi$ is an even character, which is why we always have $\Omega_f^+$ in the denominator.
\end{proof}
\begin{remark}\label{rk:nonzero}
Let $k$ denote the weight of $f$. When $k\ge3$, $L(f, \chi^{-1}, 1)\ne0$ for all $\chi$ (by the functional equation, it is a non-zero multiple of $L(f,\chi,k-1)$, which is non-zero since it can be expressed as an Euler product).  In particular, Theorem~\ref{thm:interpol} implies that $\theta_{n,f}\ne0$  for all $n\ge1$. When $k=2$, the main result of \cite{Roh} tells us that $L(\overline{f}, \bar{\chi}, 1)\ne0$ for  all but finitely many $\chi$. Thus, Theorem~\ref{thm:interpol} implies that $\theta_{n,f}\ne0$ for $n\gg0$.
\end{remark}

\section{Mazur--Tate elements of Ramanujan's tau function at {$p=2$ and $3$}}\label{S:3MT}

We  link the $\lambda$-invariants of the Mazur--Tate elements attached to $\Delta$ at $p=2$ and $3$ to those attached to certain elliptic curves defined over $\QQ$ with additive reduction at $p$. Throughout this section, we fix $p\in\{2,3\}$ and $N\in\{27,32\}$  so that $p|N$. In other words, $(p,N)=(2,32)$ or $(3,27)$. 

According to the online database LMFDB, the complex vector space $S_2(\Gamma_0(N))$ is of dimension one. Let $f_N$ denote the unique normalized cuspform in this space. All elliptic curves of conductor $N$ that are defined over $\QQ$ have to correspond to this modular form under the modularity theorem.

For both choices of $N$, there is one single isogeny class and four isomorphism classes of  elliptic curves of conductor $N$ that are defined over $\QQ$.  At least one of the isomorphism classes admit non-trivial $p$-torsions over $\QQ$. Thus, Theorem~\ref{thm:cong} says that $a_\ell(E)\equiv\tau(\ell) \mod p$ for all $\ell\ne p$. Recall that $\tau(p)\equiv 0\mod p$. Furthermore, since $E$ has additive reduction at $p$, we have $a_p(E)=0$. Thus,
\[
\tau(p)\equiv a_p(E)\mod p.
\]
Consequently, all Fourier coefficients of $\Delta$ and $f$ are congruent modulo $p$ and  we have 
\begin{equation}
f\equiv \Delta\mod p
\label{eq:cong-Delta-3}    
\end{equation}
as modular forms. 
\begin{remark}
We thank the referee for pointing out to us that \eqref{eq:cong-Delta-3} may be obtained directly in the following way. Let $$\eta(z)=q^{\frac{1}{24}}\prod_{n\ge1}(1-q^n)$$ be the Dedekind eta function. Then,
\begin{align*}
f_{32}=\eta(4z)^2\eta(8z)^2=q\prod_{n\ge1}(1-q^{4n})^2(1-q^{8n})^2\equiv q\prod(1-q^n)^{24}=\Delta\mod 2;\\
f_{27}=\eta(3z)^2\eta(9z)^2=q\prod_{n\ge1}(1-q^{3n})^2(1-q^{9n})^2\equiv q\prod(1-q^n)^{24}=\Delta\mod 3.    
\end{align*}
\end{remark}

The congruence \eqref{eq:cong-Delta-3} allows us to prove Theorem~\ref{thm:B}:

\begin{theorem}\label{thm:Delta3}
Let $(p,N)=(2,32)$ or $(3,27)$, $n\ge1$ and $E$ an elliptic curve defined over $\QQ$ of conductor $N$. If $\mu(\theta_{n,\Delta})=0$, then
\[
\lambda(\theta_{n,\Delta})=\lambda(\theta_{n,E}).
\]
\end{theorem}

\begin{proof}
 We define $\alpha'$ to be the composition
 \[
 H^1_c(\Gamma_0(1),V_{10}(\Zp))\stackrel{\alpha}\rightarrow H^1_c(\Gamma_0(N),\Zp/N\Zp)\rightarrow H^1_c(\Gamma_0(N),\Fp),
 \]
where $\alpha$ is the Hecke equivariant map defined as in \cite[\S7]{PW} and the second arrow is given by the natural projection map.

Recall that  $f_N\in S_2(\Gamma_0(N))$ denotes the  modular form  corresponding to $E$. Let us write $\ovp_\Delta\in H^1_c(\Gamma_0(1),V_{10}(\Fp))$ (resp. {$\ovp_{f_N}\in H^1_c(\Gamma_0(p^r),\Fp)$}) for the image of $\vp_\Delta/\Omega_\Delta^+$ (resp. {$\vp_{f_N}/\Omega_{f_N}^+$}) modulo $p$. {The choice of periods  ensures that $\ovp_\Delta$ and $\ovp_{f_N}$} are both non-zero (see \cite[Definition~2.1]{PW}).
 
 Since we have assumed that $\theta_{n,\Delta}\not\in p\Zp[\cG_n]$ and that
 \[
 \vartheta_n(\alpha'(\ovp_\Delta))\equiv \theta_{n,\Delta}\mod p\Zp[G_n]
 \]
 by  \cite[Lemma~4.6]{PW} (here, $\vartheta$ is defined as in \cite[(2) on P.357]{PW}), it follows  that $\alpha'(\ovp_\Delta)\ne0$.  Furthermore, as $\alpha'$ is Hecke equivariant, \eqref{eq:cong-Delta-3} implies that
  $$\alpha'(\ovp_\Delta) = \ovp_{f_N}.$$  Hence, applying \cite[Lemma~4.6]{PW} once more gives
  \[
  \theta_{n,\Delta}\equiv \theta_{n,f}\not\equiv 0\mod p\Zp[\cG_n],
  \]
 from which the desired equality of $\lambda$-invariants follows.
\end{proof}
When $p=2$, our numerical investigations show that  $\mu(\theta_{n,\Delta})$ is not always zero. Nonetheless, it turns out that the conclusion of Theorem~\ref{thm:Delta3} still holds for $n>3$, with
$$\lambda(\theta_{n,\Delta}) = \lambda(\theta_{n,E}) = 2^{n-1} - 2.$$
When $p=3$, our numerical investigations have led us to believe that the hypothesis $\theta_{n,\Delta}\notin p\Zp[\cG_n]$ is always true. We have found that $$\lambda(\theta_{n,\Delta}) = \lambda(\theta_{n,E}) = 3^n - 2$$ for all values of $n$ that we have studied (see \S6).

\section{Mazur--Tate elements at additive primes}
Let $E/ \QQ$ be an elliptic curve having additive reduction at a fixed prime $p$. Other than the the settings where the works of Delbourgo \cite{del-compositio,del-JNT} apply, it is not known how to define a $p$-adic L-function that would interpolate the complex $L$-values of $E$. Nonetheless, it is possible to compute $p$-adic Mazur--Tate elements of level $n$ attached to $E$ as given in Definition~\ref{def:MT}. Interestingly, the calculations we made (see \S\ref{sec:data}) show that the lambda invariants of such elements behave in a surprisingly regular manner, even though  we do not know whether such patterns can be explained using Iwasawa-theoretic objects. In this section, we  prove a  theoretical lower bound on these lambda invariants (see Corollary~\ref{cor:lambda-add} below). We note that this lower bound is attained by the curves     $45a, 63a, 72a,90c,99a,99b,99d$ when $p=3$, $150a $ when $p=5$ and $147c,294b$ when $p=7$ (see Tables \ref{tab:table1}-\ref{tab:table3} in \S\ref{sec:data}).

We  recall the following norm relation satisfied by the Mazur--Tate elements.

\begin{defn}\label{def:pi}
 We denote by $\cor_n^{n+1}: \ZZ[G_{n+1}] \to \ZZ[G_{n}]$ the natural projection map.
\end{defn}

\begin{lemma}\label{lem:tau27}
Let $E/\QQ$ be an elliptic curve of conductor $N_E$ and denote by $\theta_{n,E}$ its associated Mazur--Tate element as given in Definition~\ref{def:MT}.  If $p \big | N_E$ and $m\ge 1$, then $$ \cor_{m}^{m+1} (\theta_{m+1}) = a_{p}(E) \cdot \theta_m .$$
\end{lemma}
\begin{proof}
See \cite[\S1.3]{MT}.
\end{proof}
In the case of $p$ being an additive prime, it has the following consequence on the $\lambda$-invariant of the Mazur--Tate elements.
\begin{corollary}\label{cor:lambda-add}
Let $E/\QQ$ be an elliptic curve with additive reduction at $p$. Then $$\lambda(\theta_{n,E}) \ge p^{n-1}$$
for all $n\ge1$.
\end{corollary}
\begin{proof}
 Lemma~\ref{lem:tau27} tells us that that $\cor_n^{n+1}(\theta_n(f_{E})) = 0$ for all $n \geq 1$ since $a_p(E)=0$ when $E$ has additive reduction at $p$. This implies that $\theta_{n+1} = g_n \cdot \omega_n${, where $\omega_n=(1+X)^{p^n}-1$ and}  $g_n \in \ZZ_p [{X} ]$. So, $\lambda(\theta_{n+1}) = \lambda(g_n) + p^n\ge p^n$ as required.
\end{proof}
The same is true for the Mazur--Tate elements attached to $\Delta$ at  $p=3$:
\begin{corollary}
At $p=2$ or $3$, if $\mu(\theta_{n,\Delta})=0$, we have
\[
\lambda(\theta_{n,\Delta})\ge p^{n-1}
\]
for all $n\ge1$.
\end{corollary}
\begin{proof}
This follows from Theorem~\ref{thm:Delta3} and Corollary~\ref{cor:lambda-add}.
\end{proof}

\section{Numerical data}\label{sec:data}
In this section, we present a brief summary of the numerical results we have obtained. In Tables~\ref{tab:table0}, \ref{tab:table1}, \ref{tab:table2} and \ref{tab:table3}, we give the $\lambda$-invariants of Mazur--Tate elements attached to elliptic curves having additive reduction at a fixed prime $2$, $3$, $5$ and $7$ that we have computed respectively. We have found very uniform behaviour of these invariants. It seems to suggest that the Mazur--Tate elements we computed might be related to certain bounded $p$-adic $L$-functions attached to these elliptic curves. We plan to study this in a future project.

The following tables contain our computations of the Iwasawa $\lambda$-invariants of $p$-adic Mazur--Tate elements of level $n$ attached to elliptic curves having additive reduction at $p$. Since the Mazur--Tate elements are the same up to multiplication by a scalar for all elliptic curves in a given isogeny class, we organize  our data by isogeny class using Cremona label. In the last column, we indicate our predictions for $\lambda(\theta_m)$ for $m$ sufficiently large according to the values we computed.

The calculations we did were carried out on Sage modifying slightly Pollack's algorithm, available on \url{https://github.com/rpollack9974/OMS}. The codes we used for our computations can  be found at \url{https://github.com/anthonydoyon/Ramanujan-s-tau-and-MT-elts}.

In what follows, we write
$$q_m = \left\{
                \begin{array}{ll}
                  p^{m-1} - p^{m-2} + \dots + p - 1 \quad \text{if } m \text{ is even}\\
                  p^{m-1} - p^{m-2} + \dots + p^2 - p \quad \text{if } m \text{ is odd.}\\
                \end{array}
              \right. $$
For some specific isogeny classes, for instance, $153a, 153c, 225a$ and $225b$,  we could not find one single formula for $\lambda(\theta_m)$ in terms of $m$, but rather, two separate formulas depending on the parity of $m$, involving $q_m$. The term $q_m$ appears naturally for elliptic curves with good supersingular reduction at $p$  where the $\theta$ elements are related to Pollack's plus and minus $p$-adic $L$-functions defined in \cite{pollack03} (see \cite[\S4.1]{PW}). As given in Tables~\ref{tab:table5} and \ref{tab:table6}, these curves all have potentially supersingular reduction at $p$. This  suggests that the Mazur--Tate elements for these curves might be related to Pollack's plus and minus $p$-adic $L$-functions. Curiously, there are certain curves with potentially supersingular reduction whose Mazur--Tate elements do not exhibit such patterns. We will look for a theoretic explanation on how these two distinct cases arise in our follow-up project. 

\clearpage
\begin{table}[h!]
  \begin{center}
    \caption{$\lambda$-invariants for $p$-adic Mazur--Tate elements of level $n$ of some elliptic curves with additive reduction at $p = 2$.}
    \label{tab:table0}
    \begin{tabular}{l|l|l|l|l|l|l|l|l}
        \toprule
        \textbf{Isogeny class} &\boldmath $ n=1$ & \textbf{2} & \textbf{3} & \textbf{4} & \textbf{5} & \textbf{6} & \textbf{7} & \boldmath$m$\\
        \midrule
        \multirow{2}{*}{$20a$} & \multirow{2}{*}{$\infty$} & \multirow{2}{*}{$1$} & \multirow{2}{*}{$\infty$} & \multirow{2}{*}{$7$} & \multirow{2}{*}{$9$} & \multirow{2}{*}{$31$} & \multirow{2}{*}{$33$} & $2^{m-1}-1$ ($m$ even)\\
        & & & & & & & & $2^{m-2} + 1$ ($m$ odd)\\
      \midrule 
      $24a, 48a$ & $\infty$ & $1$ & $3$ & $7$ & $15$ & $31$ & $63$ & $2^{m-1} - 1$\\
      \midrule
      $32a$ & $\infty$ & $1$ & $2$ & $6$ & $14$ & $30$ & $62$ & $2^{m-1}-2$\\
      \midrule
      $36a, 56a$ & $\infty$ & $\infty$ & $2$ & $4$ & $8$ & $16$ & $32$ & $2^{m-2}$\\
      \midrule
      \multirow{2}{*}{$40a$} & \multirow{2}{*}{$\infty$} & \multirow{2}{*}{$\infty$} & \multirow{2}{*}{$3$} & \multirow{2}{*}{$\infty$} & \multirow{2}{*}{$15$} & \multirow{2}{*}{$17$} & \multirow{2}{*}{$63$} & $2^{m-2} + 1$ ($m$ even)\\
      & & & & & & & & $2^{m-1} - 1$ ($m$ odd)\\
      \midrule
      \multirow{2}{*}{$44a$} & \multirow{2}{*}{$\infty$} & \multirow{2}{*}{$1$} & \multirow{2}{*}{$3$} & \multirow{2}{*}{$5$} & \multirow{2}{*}{$11$} & \multirow{2}{*}{$21$} & \multirow{2}{*}{$43$} & $q_m$ ($m$ even)\\
      & & & & & & & & $q_m + 1$ ($m$ odd)\\
      \midrule
      \multirow{2}{*}{$52a$} & \multirow{2}{*}{$\infty$} & \multirow{2}{*}{$1$} & \multirow{2}{*}{$3$} & \multirow{2}{*}{$7$} & \multirow{2}{*}{$11$} & \multirow{2}{*}{$31$} & \multirow{2}{*}{$35$} & $2^{m-1} - 1$ ($m$ even)\\
      & & & & & & & & $2^{m-2} + 3$ ($m$ odd)\\
      \midrule
      $64a$ & $\infty$ & $1$ & $3$ & $4$ & $10$ & $22$ & $46$ & $3 \cdot 2^{m-3} - 2$\\
        \bottomrule 
    \end{tabular}
  \end{center}
\end{table}
\vspace{0.5cm}
\begin{table}[h!]
  \begin{center}
    \caption{$\lambda$-invariants for $p$-adic Mazur--Tate elements of level $n$ of some elliptic curves with additive reduction at $p = 3$.}
    \label{tab:table1}
    \begin{tabular}{l|l|l|l|l|l|l|l|l}
      \toprule 
      \textbf{Isogeny class} &\boldmath $ n=1$ & \textbf{2} & \textbf{3} & \textbf{4} & \textbf{5} & \textbf{6} & \textbf{7} & \boldmath$m$\\
      \midrule 
      $27a, 54a$ & $1$ & $7$ & $25$ & $79$ & $241$ & $727$ & $2185$ & $3^m - 2$\\
      \midrule
      $36a, 54b, 90a,$ & \multirow{2}{*}{$2$} & \multirow{2}{*}{$8$} & \multirow{2}{*}{$26$} & \multirow{2}{*}{$80$} & \multirow{2}{*}{$242$} & \multirow{2}{*}{$728$} & \multirow{2}{*}{$2186$} & \multirow{2}{*}{$3^m - 1$}\\
      $90b, 108a$ & & & & & & & &\\
      \midrule
      $45a, 63a, 72a,$ & \multirow{3}{*}{$1$} & \multirow{3}{*}{$3$} & \multirow{3}{*}{$9$} & \multirow{3}{*}{$27$} & \multirow{3}{*}{$81$} & \multirow{3}{*}{$243$} & \multirow{3}{*}{$729$} & \multirow{3}{*}{$3^{m-1}$}\\
      $90c, 99a, 99b,$ & & & & & & & &\\
      $99d$ & & & & & & & &\\
      \midrule
      $99c$ & $\infty$ & $6$ & $18$ & $54$ & $162$ & $486$ & $1458$ & $2 \cdot 3^{m-1}$\\
      \midrule
      \multirow{2}{*}{$153a$} & \multirow{2}{*}{$1$} & \multirow{2}{*}{$\infty$} & \multirow{2}{*}{$11$} & \multirow{2}{*}{$39$} & \multirow{2}{*}{$101$} & \multirow{2}{*}{$309$} & \multirow{2}{*}{$911$} & $3^{m-1} + q_{m-1} + 6$ ($m$ even)\\
      & & & & & & & & $3^{m-1} + q_{m-1}$ ($m$ odd)\\
      \midrule
      \multirow{2}{*}{$153c$} & \multirow{2}{*}{$\infty$}  & \multirow{2}{*}{$5$} & \multirow{2}{*}{$21$} & \multirow{2}{*}{$47$} & \multirow{2}{*}{$147$} & \multirow{2}{*}{$425$} & \multirow{2}{*}{$1281$} & $3^{m-1} + q_m$ ($m$ even)\\
      & & & & & & & & $3^{m-1} + q_m + 6$ ($m$ odd)\\
      \midrule
      $153d$ & $2$  & $6$ & $20$ & $60$ & $182$ & $546$ & $1640$ & $q_{m+1}$\\
      \bottomrule 
    \end{tabular}
  \end{center}
\end{table}
\clearpage
\begin{table}[h!]
  \begin{center}
    \caption{$\lambda$-invariants for $p$-adic Mazur--Tate elements of level $n$ of some elliptic curves with additive reduction at $p = 5$.}
    \label{tab:table2}
    \begin{tabular}{l|l|l|l|l|l|l}
      \toprule 
      \textbf{Isogeny class} & \boldmath $ n=1$& \textbf{2} & \textbf{3} & \textbf{4} & \textbf{5} & \boldmath$m$\\
      \midrule 
      $50b, 75c$ & $4$ & $24$ & $124$ & $624$ & $3124$ & $5^m - 1$\\
      \midrule
      $75b, 100a, 150c$ & $2$ & $10$ & $50$ & $250$ & $1250$ & $2 \cdot 5^{m-1}$\\
      \midrule
      $50a, 75a, 150b, 175c$ & $3$ & $15$ & $75$ & $375$ & $1875$ & $3 \cdot 5^{m-1}$\\
      \midrule
      $175b$ & $4$ & $12$ & $52$ & $252$ & $1252$ & $2 \cdot 5^{m-1} + 2$\\
      \midrule
      $175a$ & $2$ & $6$ & $26$ & $126$ & $626$ & $5^{m-1} + 1$\\
      \midrule
      $150a$ & $1$ & $5$ & $25$ & $125$ & $625$ & $5^{m-1}$\\
      \midrule
      \multirow{2}{*}{$225a$} & \multirow{2}{*}{$1$} & \multirow{2}{*}{$8$} & \multirow{2}{*}{$37$} & \multirow{2}{*}{$188$} & \multirow{2}{*}{$937$} & $5^{m-1} + 3 \cdot q_{m-1} + 3$ ($m$ even)\\
      & & & & & & $5^{m-1} + 3 \cdot q_{m-1}$ ($m$ odd)\\
      \midrule
      \multirow{2}{*}{$225b$} & \multirow{2}{*}{$4$} & \multirow{2}{*}{$17$} & \multirow{2}{*}{$88$} & \multirow{2}{*}{$437$} & \multirow{2}{*}{$2188$} & $3 \cdot 5^{m-1} + 3 \cdot q_{m-1} + 2$ ($m$ even)\\
      & & & & & & $3 \cdot 5^{m-1} + 3 \cdot q_{m-1} + 1$ ($m$ odd)\\
      \bottomrule 
    \end{tabular}
  \end{center}
\end{table}

\vspace{0.5cm}

\begin{table}[h!]
  \begin{center}
    \caption{$\lambda$-invariants for $p$-adic Mazur--Tate elements of level $n$ of some elliptic curves with additive reduction at $p = 7$.}
    \label{tab:table3}
    \begin{tabular}{l|l|l|l|l|l}
    \toprule 
      \textbf{Isogeny class} &\boldmath $ n=1$ & \textbf{2} & \textbf{3} & \textbf{4} & \boldmath$m$\\
      \midrule 
      $49a, 245b, 294e, 294f, 392b, 441a$ & $5$ & $35$ & $245$ & $1715$ & $5 \cdot 7^{m-1}$\\
      \midrule
      $98a, 147a, 294c, 392d$ & $3$ & $21$ & $147$ & $1029$ & $3 \cdot 7^{m-1}$\\
      \midrule
      $147b, 196b, 294a, 392e, 441e$ & $4$ & $28$ & $196$ & $1372$ & $4 \cdot 7^{m-1}$\\
      \midrule
      $147c, 294b$ & $1$ & $7$ & $49$ & $343$ & $7^{m-1}$\\
      \midrule
      $245a, 294d, 294g, 441d$ & $2$ & $14$ & $98$ & $686$ & $2 \cdot 7^{m-1}$\\
      \midrule
      $196a, 392f$ & $2$ & $8$ & $50$ & $344$ & $7^{m-1} + 1$\\
      \midrule
      $245c, 392a, 441c$ & $4$ & $22$ & $148$ & $1030$ & $3 \cdot 7^{m-1} + 1$\\
      \midrule
      $392c, 441b$ & $3$ & $15$ & $99$ & $687$ & $2 \cdot 7^{m-1} + 1$\\
      \midrule
      $441f$ & $3$ & $9$ & $51$ & $345$ & $7^{m-1}+2$\\
    \bottomrule 
    \end{tabular}
  \end{center}
\end{table}

\clearpage

In the following table, we give the values of $\lambda$-invariants of the Mazur--Tate elements attached to $\Delta$ at the non-ordinary primes $p=3,5,7$ that we have been able to compute, We note that the values for $p = 3$ agree with those for the isogeny classes 27a and 54a in Table~\ref{tab:table1}, as predicted by Theorem~\ref{thm:Delta3}.

\vspace{0.5cm}

\begin{table}[h!]
  \begin{center}
    \caption{$\lambda$-invariants of $p$-adic Mazur--Tate elements of level $n$ associated to $\Delta$.}
    \label{tab:table4}
    \begin{tabular}{l|l|l|l|l|l|l|l|l}
    \toprule 
      \boldmath$p$ & \boldmath $ n=1$ & \textbf{2} & \textbf{3} & \textbf{4} & \textbf{5} & \textbf{6} & \textbf{7} & \boldmath$m$\\
      \midrule
      $2$ & $0$ & $1$ & $3$ & $6$ & $14$ & $30$ & $62$ & $2^{m-1} - 2$\\
      \midrule 
      $3$ & $1$ & $7$ & $25$ & $79$ & $241$ & $727$ & ----- & $3^m - 2$\\
      \midrule
      $5$ & $4$ & $24$ & $124$ & $624$ & ----- & ----- & ----- & $5^m - 1$\\
      \midrule
      $7$ & $6$ & $48$ & $342$ & ----- & ----- & ----- & ----- & $7^m - 1$\\
    \bottomrule 
    \end{tabular}
  \end{center}
\end{table}

If $E/\QQ$ is an  elliptic curve with additive reduction at $p$, we recall it has either potentially good or potentially multiplicative reduction at $p$. We have observed that the formulae of $\lambda$-invariants we have found in Tables~\ref{tab:table1} to \ref{tab:table3} seem to be related to the potential reduction type of the curves. We give in Tables~\ref{tab:table5}, \ref{tab:table6} and \ref{tab:table7} the potential reduction types that we have been able to work out for the curves we have studied with $p=3,5$ and $7$ respectively. When $p=3$ and $E=54a$ or $54b$, we have found that these curves have potentially good reduction. But we have been unable to find the number field where good reduction is attained. As a result, we do not know whether it has potentially good ordinary reduction or potentially good supersingular reduction. It seems even more difficult to determine the potential type of an elliptic curve with additive reduction at $p=2$. Since we are mostly interested in further investigations in Iwasawa Theory for odd primes, we have decided not to study the potential reduction types of the elliptic curves in Table~\ref{tab:table0}.

\begin{table}[h!]
  \begin{center}
    \caption{Potential reduction of elliptic curves having additive reduction at $p=3$.}
    \label{tab:table5}
    \begin{tabular}{l|l|l}
    \toprule \multirow{2}{*}{\textbf{Isogeny class}} & \multirow{2}{*}{\textbf{\boldmath $K = \QQ( \cdot )$}} & \textbf{Reduction of} \boldmath $E / K$ \textbf{at a prime}\\
    & & \textbf{ideal of} \boldmath $K$ \textbf{lying above} \boldmath $p$\\
    \midrule
    $27a$ & $54^{\frac{1}{12}}$ & supersingular\\
    \midrule
    $36a$ & $3^{\frac{1}{4}}$ & supersingular\\
    \midrule
    $45a$ & $3^{\frac{1}{2}}$ & split multiplicative\\
    \midrule
    $54a$ & ?? & good\\
    \midrule
    $54b$ & ?? & good\\
    \midrule
    $63a$ & $3^{\frac{1}{2}}$ & non-split multiplicative\\
    \midrule
    $72a$ & $3^{\frac{1}{2}}$ & split multiplicative\\
    \midrule
    $90a$ & $3^{\frac{1}{4}}$ & supersingular\\
    \midrule
    $90b$ & $3^{\frac{1}{4}}$ & supersingular\\
    \midrule
    $90c$ & $3^{\frac{1}{2}}$ & non-split multiplicative\\
    \midrule
    $99a$ & $3^{\frac{1}{4}}$ & supersingular\\
    \midrule
    $99b$ & $3^{\frac{1}{2}}$ & split multiplicative\\
    \midrule
    $99c$ & $3^{\frac{1}{4}}$ & supersingular\\
    \midrule
    $99d$ & $3^{\frac{1}{2}}$ & good ordinary\\
    \midrule
    $108a$ & $54^{\frac{1}{12}}$ & supersingular\\
    \midrule
    $153a$ & $6^{\frac{1}{4}}$ & supersingular\\
    \midrule
    $153c$ & $3^{\frac{1}{2}}$ & supersingular\\
    \midrule
    $153d$ & $3^{\frac{1}{4}}$ & supersingular\\
    \bottomrule 
    \end{tabular}
  \end{center}
\end{table}

\begin{table}[h!]
  \begin{center}
    \caption{Potential reduction of elliptic curves having additive reduction at $p=5$.}
    \label{tab:table6}
    \begin{tabular}{l|l|l}
    \toprule \multirow{2}{*}{\textbf{Isogeny class}} & \multirow{2}{*}{\textbf{\boldmath $K = \QQ( \cdot )$}} & \textbf{Reduction of} \boldmath $E / K$ \textbf{at a prime}\\
    & & \textbf{ideal of} \boldmath $K$ \textbf{lying above} \boldmath $p$\\
    \midrule
    $50a$ & $5^{\frac{1}{3}}$ & supersingular\\
    \midrule
    $50b$ & $5^{\frac{1}{6}}$ & supersingular\\
    \midrule
    $75a$ & $5^{\frac{1}{3}}$ & supersingular\\
    \midrule
    $75b$ & $5^{\frac{1}{2}}$ & split multiplicative\\
    \midrule
    $75c$ & $5^{\frac{1}{6}}$ & supersingular\\
    \midrule
    $100a$ & $5^{\frac{1}{2}}$ & non-split multiplicative\\
    \midrule
    $150a$ & $5^{\frac{1}{4}}$ & good ordinary\\
    \midrule
    $150b$ & $5^{\frac{1}{4}}$ & good ordinary\\
    \midrule
    $150c$ & $5^{\frac{1}{2}}$ & non-split multiplicative\\
    \midrule
    $175a$ & $5^{\frac{1}{4}}$ & good ordinary\\
    \midrule
    $175b$ & $5^{\frac{1}{2}}$ & non-split multiplicative\\
    \midrule
    $175c$ & $5^{\frac{1}{4}}$ & good ordinary\\
    \midrule
    $225a$ & $5^{\frac{1}{6}}$ & supersingular\\
    \midrule
    $225b$ & $5^{\frac{1}{3}}$ & supersingular\\
    \bottomrule 
    \end{tabular}
  \end{center}
\end{table}
\newpage

\begin{table}[h!]
  \begin{center}
    \caption{Potential reduction of elliptic curves having additive reduction at $p=7$.}
    \label{tab:table7}
    \begin{tabular}{l|l|l}
    \toprule \multirow{2}{*}{\textbf{Isogeny class}} & \multirow{2}{*}{\textbf{\boldmath $K = \QQ( \cdot )$}} & \textbf{Reduction of} \boldmath $E / K$ \textbf{at a prime}\\
    & & \textbf{ideal of} \boldmath $K$ \textbf{lying above} \boldmath $p$\\
    \midrule
    $49a$ & $7^{\frac{1}{4}}$ & supersingular\\
    \midrule
    $98a$ & $7^{\frac{1}{2}}$ & non-split multiplicative\\
    \midrule
    $147a$ & $7^{\frac{1}{2}}$ & split multiplicative\\
    \midrule
    $147b$ & $7^{\frac{1}{3}}$ & good ordinary\\
    \midrule
    $147c$ & $7^{\frac{1}{6}}$ & good ordinary\\
    \midrule
    $196a$ & $7^{\frac{1}{6}}$ & good ordinary\\
    \midrule
    $196b$ & $7^{\frac{1}{3}}$ & good ordinary\\
    \midrule
    $245a$ & $7^{\frac{1}{4}}$ & supersingular\\
    \midrule
    $245b$ & $7^{\frac{1}{4}}$ & supersingular\\
    \midrule
    $245c$ & $7^{\frac{1}{2}}$ & non-split multiplicative\\
    \midrule
    $294a$ & $7^{\frac{1}{3}}$ & good ordinary\\
    \midrule
    $294b$ & $7^{\frac{1}{6}}$ & good ordinary\\
    \midrule
    $294c$ & $7^{\frac{1}{2}}$ & split multiplicative\\
    \midrule
    $294d$ & $7^{\frac{1}{3}}$ & good ordinary\\
    \midrule
    $294e$ & $7^{\frac{1}{6}}$ & good ordinary\\
    \midrule
    $294f$ & $7^{\frac{1}{4}}$ & supersingular\\
    \midrule
    $294g$ & $7^{\frac{1}{4}}$ & supersingular\\
    \midrule
    $392a$ & $7^{\frac{1}{2}}$ & split multiplicative\\
    \midrule
    $392b$ & $7^{\frac{1}{6}}$ & good ordinary\\
    \midrule
    $392c$ & $7^{\frac{1}{3}}$ & good ordinary\\
    \midrule
    $392d$ & $7^{\frac{1}{2}}$ & non-split multiplicative\\
    \midrule
    $392e$ & $7^{\frac{1}{3}}$ & good ordinary\\
    \midrule
    $392f$ & $7^{\frac{1}{6}}$ & good ordinary\\
    \midrule
    $441a$ & $7^{\frac{1}{6}}$ & good ordinary\\
    \midrule
    $441b$ & $7^{\frac{1}{3}}$ & good ordinary\\
    \midrule
    $441c$ & $7^{\frac{1}{2}}$ & split multiplicative\\
    \midrule
    $441d$ & $7^{\frac{1}{4}}$ & supersingular\\
    \midrule
    $441e$ & $7^{\frac{1}{3}}$ & good ordinary\\
    \midrule
    $441f$ & $7^{\frac{1}{6}}$ & good ordinary\\
    \bottomrule 
    \end{tabular}
  \end{center}
\end{table}

\clearpage
\bibliographystyle{amsalpha}
\bibliography{references}
\end{document}